\documentclass[10pt, a4paper]{amsart}

\usepackage[hmargin=3.5cm,vmargin=3.5cm]{geometry}
\usepackage[backend=biber, style=alphabetic, url=false]{biblatex}
\AtEveryBibitem{\clearlist{language}}
\usepackage{amsmath}
\usepackage{amssymb}
\usepackage{amsthm}
\usepackage{mathrsfs}
\usepackage{mathtools}
\usepackage{url}
\usepackage{graphicx}
\usepackage{enumerate}

\usepackage{dsfont}
\usepackage[colorlinks,citecolor=blue,linkcolor=blue,urlcolor=blue,filecolor=blue]{hyperref}
\usepackage[capitalize,nameinlink,noabbrev]{cleveref}
\usepackage{aliascnt}
\RequirePackage{doi}

\theoremstyle{plain}

\newtheorem*{thma}{Theorem A}

\newtheorem*{thmb}{Theorem B}

\newtheorem*{thmc}{Theorem C}

\newtheorem*{thmd}{Theorem D}

\theoremstyle{remark}

\newtheorem*{rem}{Remark}

\theoremstyle{plain}

\newtheorem{theorem}{Theorem}
\newtheorem{proposition}[theorem]{Proposition}

\newtheorem{corollary}[theorem]{Corollary}

\newtheorem{lemma}[theorem]{Lemma}

\newtheorem{observation}[theorem]{Observation}

\theoremstyle{definition}

\newtheorem{definition}[theorem]{Definition}

\theoremstyle{remark}

\newtheorem{remark}[theorem]{Remark}

\newcommand{\tl}{\mathcal{TL}}
\newcommand{\frakd}{\mathfrak{D}}
\newcommand{\trivial}{\mathds{1}}
\newcommand{\setn}{\langle n - 1 \rangle}
\newcommand{\n}{\mathbb{N}}
\DeclareMathOperator{\tor}{Tor}
\DeclareMathOperator{\cupmod}{Cup}

\title{The homology of a Temperley--Lieb algebra on an odd number of strands}

\author{Robin J. Sroka}
\address{Department of Mathematics \& Statistics, McMaster University, 
Hamilton, ON, Canada}
\email{srokar@mcmaster.ca}

\begin{document}

\begin{abstract}
  We show that the homology of any Temperley--Lieb algebra $\tl_n(a)$ on an odd number of strands vanishes in positive degrees. This improves a result obtained by Boyd--Hepworth. In addition we present alternative arguments for the following two vanishing results of Boyd--Hepworth. (1) The stable homology of Temperley--Lieb algebras is trivial. (2) If the parameter $a \in R$ is a unit, then the homology of any Temperley--Lieb algebra is concentrated in degree zero.
\end{abstract}

\maketitle

\section{Introduction}

Let $a \in R$ be an element in a commutative unital ring. Intuitively, the Temperley--Lieb algebra $\tl_n(a)$ on $n$ strands is the $R$-algebra, whose underlying $R$-module has a basis given by isotopy classes of planar diagrams. The multiplication of two planar diagrams is given by gluing them together. Any circles appearing in the resulting diagram correspond to multiplication by $a \in R$. This is illustrated in \cref{fig:temperleyliebmultiplication}, and a precise definition is given in the next section (see \cref{temperleyliebalgebras}).

\begin{figure}[h]
  \centering
  \def\svgscale{.9}
\begingroup%
  \makeatletter%
  \providecommand\color[2][]{%
    \errmessage{(Inkscape) Color is used for the text in Inkscape, but the package 'color.sty' is not loaded}%
    \renewcommand\color[2][]{}%
  }%
  \providecommand\transparent[1]{%
    \errmessage{(Inkscape) Transparency is used (non-zero) for the text in Inkscape, but the package 'transparent.sty' is not loaded}%
    \renewcommand\transparent[1]{}%
  }%
  \providecommand\rotatebox[2]{#2}%
  \newcommand*\fsize{\dimexpr\f@size pt\relax}%
  \newcommand*\lineheight[1]{\fontsize{\fsize}{#1\fsize}\selectfont}%
  \ifx\svgwidth\undefined%
    \setlength{\unitlength}{411.02362205bp}%
    \ifx\svgscale\undefined%
      \relax%
    \else%
      \setlength{\unitlength}{\unitlength * \real{\svgscale}}%
    \fi%
  \else%
    \setlength{\unitlength}{\svgwidth}%
  \fi%
  \global\let\svgwidth\undefined%
  \global\let\svgscale\undefined%
  \makeatother%
  \begin{picture}(1,0.34482759)%
    \lineheight{1}%
    \setlength\tabcolsep{0pt}%
    \put(0,0){\includegraphics[width=\unitlength,page=1]{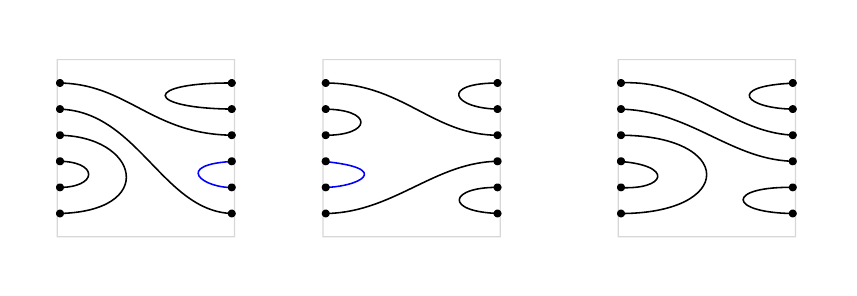}}%
    \put(0.62803804,0.15667219){\color[rgb]{0,0,0}\makebox(0,0)[lt]{\lineheight{1.25}\smash{\begin{tabular}[t]{l}=\hspace{5pt}{\color{blue} $a$} \textperiodcentered\end{tabular}}}}%
    \put(0.32079022,0.1568152){\color[rgb]{0,0,0}\makebox(0,0)[lt]{\lineheight{1.25}\smash{\begin{tabular}[t]{l}\textperiodcentered\end{tabular}}}}%
  \end{picture}%
\endgroup%

  \caption{Multiplication of two planar diagrams in $\tl_6(a)$.}
  \label{fig:temperleyliebmultiplication}
\end{figure}

The Temperley--Lieb algebra $\tl_n(a)$ has a natural augmentation $\epsilon_n: \tl_n(a) \to R$ that maps all non-identity planar diagrams to zero. The augmentation gives $R$ the structure of a $\tl_n(a)$-module, which we call the trivial module and denote by $\trivial$. In this paper, we examine the homology of Temperley--Lieb algebras with coefficients in this module
$$H_\star(\tl_n(a), \trivial) = \tor^{\tl_n(a)}_\star(\trivial, \trivial).$$
Recently, these homology groups have been studied by Boyd--Hepworth \cite{boydhepworth2020temperleylieb} and Randal-Williams \cite{randalwilliams2021aremarkonthehomologyoftemperleyliebalgebras}. Our new main finding is the following theorem, which removes all conditions in the vanishing results for the homology of Temperley--Lieb algebras on an odd number of strands obtained by Boyd--Hepworth.

\begin{thma}
  Let $R$ be a commutative unital ring, $a \in R$ and $n \in \n$ be odd. Consider the Temperley--Lieb algebra $\tl_n(a)$ on $n$ strands. Then, $$H_0(\tl_n(a), \trivial) = R \text{ and } H_\star(\tl_n(a), \trivial) = 0 \text{ for } \star > 0.$$
\end{thma}

\textsf{Previously known results.} Assuming that the parameter $a \in R$ used to define $\tl_n(a)$ is of the form $a = v + v^{-1}$ for a unit $v \in R$, Boyd--Hepworth proved that the homology of a Temperley--Lieb algebra on an odd number of $n$ strands is trivial in homological degrees $1 \leq \star \leq n-1$ \cite[Theorem B]{boydhepworth2020temperleylieb}. Theorem D of \cite{boydhepworth2020temperleylieb} furthermore establishes that the homology groups $H_\star(\tl_n(a), \trivial)$ are trivial in all positive degrees $\star > 0$, if $n = 2k+1$ is odd, the parameter $a = v + v^{-1} = 0$ is zero, and $R$ is a field whose characteristic does not divide $\binom{k}{r}$ for any $1 \leq r \leq k$. In the document \cite{randalwilliams2021aremarkonthehomologyoftemperleyliebalgebras}, Randal-Williams builds on the work of Boyd--Hepworth and shows that the assumptions on the parameter $a \in R$ can be removed using base change techniques \cite[Theorem B']{randalwilliams2021aremarkonthehomologyoftemperleyliebalgebras}. However, an invertibility condition on the binomials $\binom{k}{r}$ remains \cite[item $(ii)$ before Corollary 3.2]{randalwilliams2021aremarkonthehomologyoftemperleyliebalgebras}. \medskip

The proof of Theorem A presented here removes all of these conditions and gives an alternative argument for the vanishing results due to Boyd--Hepworth, as well as the strengthening obtained by Randal-Williams. \medskip

In addition to Theorem A, our methods allow us to prove a vanishing line for the homology of Temperley--Lieb algebras on an even number of strands. This vanishing line is weaker than the one obtained by Boyd--Hepworth \cite[Theorem B]{boydhepworth2020temperleylieb}. Boyd--Hepworth prove a slope $1$ vanishing line, the one we prove in this article is of slope $\frac{1}{2}$.

\begin{thmb}
  Let $R$ be a commutative unital ring, $a \in R$ and $n \in \n$ be even. Consider the Temperley--Lieb algebra $\tl_n(a)$ on $n$ strands. Then, $H_0(\tl_n(a), \trivial) = R$,
  $$H_\star(\tl_n(a), \trivial) = 0 \text{ for } 0 < \star < \frac{n}{2}$$
  and
  $$H_{\star + \frac{n}{2}} (\tl_n(a), \trivial) \cong H_{\star}(\tl_n(a), \cupmod(M))$$
  for $\star \geq 0$, where $\cupmod(M)$ is a certain $\tl_n(a)$-module (see \cref{cupmodules} and \cref{definitionofm}).
\end{thmb}

The coefficient $\tl_n(a)$-module $\cupmod(M)$ in Theorem B is, in general, not isomorphic to the Fineberg module appearing in the description of the higher homology groups obtained by Boyd--Hepworth \cite[Theorem 5.1]{boydhepworth2020temperleylieb}.

\begin{rem}
  In contrast with Temperley--Lieb algebras on an odd number of strands, the homology of a Temperley--Lieb algebras on an even number of $n$ strands can be non-trivial in positive degrees. This was established in Theorem C of \cite{boydhepworth2020temperleylieb}, which states that $H_{n-1}(\tl_n(a), \trivial)$ is non-zero if $n$ is even and $a = v + v^{-1} \in R$ is not a unit.
\end{rem}

Together, Theorem A and B give an alternative proof of the fact that the family of Temperley--Lieb algebras $\{\tl_n(a)\}_{n \in \n}$ satisfies homological stability and that the stable homology is trivial
$$\operatorname{colim}_{n \to \infty} H_\star(\tl_n(a), \trivial) = \begin{cases}  R, & \text{ if } \star = 0,\\ 0, & \text{ if } \star > 0. \end{cases}$$
This also follows from the work of Boyd--Hepworth and is discussed in their paper \cite{boydhepworth2020temperleylieb}. Homological stability questions for groups play a fundamental role in algebraic topology and algebraic K-theory, e.g.\ \cite{MR514863, MR569072, MR586429, MR786348, charney1987ageneralizationofatheoremofvogtmann, MR2113904, MR2220689, MR2367024, MR3572348, MR4028513}. This article can be seen as a contribution to the set of ideas formulated by Hepworth \cite{hepworth2020homologicalstabilityforiwahoriheckealgebras}, Boyd--Hepworth \cite{boydhepworth2020temperleylieb} and Boyd--Hepworth--Patzt \cite{boydhepworthpatzt2020thehomologyofthebraueralgebras} that aim to extend these techniques to families of abstract algebras.\medskip

Building on Theorem A and B we can moreover formulated an alternative argument for \cite[Theorem A]{boydhepworth2020temperleylieb}, which shows that the homology of any Temperley--Lieb algebra is trivial in positive degrees if the parameter $a \in R$ is a unit.

\begin{thmc}[{\cite{boydhepworth2020temperleylieb}, Boyd--Hepworth}]
  Let $R$ be a commutative unital ring, $a \in R$ and $n \in \n$. Consider the Temperley--Lieb algebra $\tl_n(a)$ on $n$ strands. If $a \in R$ is a unit, then
  $$H_0(\tl_n(a), \trivial) = R \text{ and } H_\star(\tl_n(a), \trivial) = 0 \text{ for } \star > 0.$$
\end{thmc}

\textsf{Comparison with work of Boyd--Hepworth \cite{boydhepworth2020temperleylieb} and Boyd \cite{boyd2020thelowdimensionalhomology}.} The general strategy of proof that we employ is similar to the one used by Boyd--Hepworth, i.e.\ we construct a certain highly connected $\tl_n(a)$-chain complex and study spectral sequences attached to it. The chain complex that we use is different from the complex of planar injective words studied by Boyd--Hepworth \cite{boydhepworth2020temperleylieb}. Indeed, we will use the ``cellular Davis complex'' $\tl \frakd$ of the Temperley--Lieb algebra $\tl_n(a)$. The structure of the complex $\tl\frakd$ is ``sensitive'' to the question whether the number of strings of a Temperley--Lieb algebra is even or odd. This is what enables us to prove Theorem A above. The chain complex $\tl \frakd$ has a simple diagrammatic description and can be seen as an algebraic analogue of the classical Davis complex of the symmetric group equipped with the ``Coxeter cell'' CW-structure \cite[Chapter 7 and 8]{davis2008thegeometryandtopologyofcoxetergroups}. The classical Davis complex of a finite symmetric group $S_n$ is a contractible CW-complex with $S_n$-action. We prove the following Temperley--Lieb analogue.

\begin{thmd}
  The ``cellular Davis complex'' $\tl\frakd$ of a Temperley--Lieb algebra (\cref{cellulardaviscomplextl}) is a contractible $\tl_n(a)$-chain complex with $H_0(\tl \frakd) = \trivial$.
\end{thmd}

From this perspective, the content of this paper is inspired by and in direct analogy with the approach that Boyd employed in \cite{boyd2020thelowdimensionalhomology} to derive formulas for the low-dimensional homology of Coxeter groups. The first two chapters of the author's thesis \cite{sroka2021thesis} make this analogy precise and explain how the chain complex $\tl \frakd$ was discovered via this analogy. \medskip

\noindent \textbf{Outline.} In \cref{sec:temperley-lieb-algebras}, we define the Temperley--Lieb algebra $\tl_n(a)$ and its trivial module $\trivial$. At the beginning of \cref{sec:vanishing-theorems-for-the-homology-of-tl-algebras}, we introduce a class of left submodules $\cupmod(F)$ of $\tl_n(a)$ which will play an important role in the rest of this article. In \cref{subsec:the-cellular-davis-complex-of-tl}, we use these modules to define the ``cellular Davis complex'' $\tl \frakd$ of a Temperley--Lieb algebra and establish that $\tl \frakd$ is contractible (Theorem D). In \cref{subsec:homology-with-coefficients-in-cup-modules}, we study the homology of the Temperley--Lieb algebra $\tl_n(a)$ with coefficients in $\cupmod(F)$. These homology groups will appear on the $E_1$-page of the spectral sequences that we use in \cref{subsec:proof-of-theorem-a-b-c} to formulate the proofs of Theorem A, B and C. \medskip

\noindent \textbf{Acknowledgements.} This article is based on the second chapter of my PhD thesis \cite{sroka2021thesis} written at the University of Copenhagen under the direction of Nathalie Wahl. All results presented here are contained in  \cite{sroka2021thesis}. I am grateful to Nathalie Wahl, Richard Hepworth and Rachael Boyd for their constructive feedback and many enlightening discussions while completing this work. I would like to thank Richard Hepworth for fruitful conversations about related ideas \cites{hepworth2020homologicalstabilityforiwahoriheckealgebras}{boyd2020thelowdimensionalhomology}[Chapter 3]{sroka2021thesis} that eventually inspired this project. For clarifying conversations and helpful comments, I would like to thank Jonas Stelzig. I am grateful to the anonymous referee for their suggestions and careful reading of this work. This research was primarily supported by the European Research Council (ERC) under the European Union’s Horizon 2020 research and innovation programme (grant agreement No.772960), as well as by the Danish National Research Foundation through the Centre for Symmetry and Deformation (DNRF92) and the Copenhagen Centre for Geometry and Topology (DNRF151). It was partially supported by NSERC Discovery Grant A4000 in connection with a Postdoctoral Fellowship at McMaster University, and by the Swedish Research Council under grant no.\ 2016-06596 while the author was in residence at Institut Mittag-Leffler in Djursholm, Sweden during the semester \emph{Higher algebraic structures in algebra, topology and geometry}. I gratefully acknowledge this support.

\section{Temperley--Lieb algebras} \label{sec:temperley-lieb-algebras}

This section contains necessary background knowledge on Temperley--Lieb algebras that we mainly learned from Kassel--Turaev \cite{kasselturaev2008braidgroups} and the exposition in Boyd--Hepworth \cite{boydhepworth2020temperleylieb}. In 1971, Temperley--Lieb introduced these algebras in their article \cite{temperleylieb1971relationsbetween}.

\begin{definition}[{\cite[Definition 5.24]{kasselturaev2008braidgroups}, \cite[Definition 2.1]{boydhepworth2020temperleylieb}}]\label{temperleyliebalgebras}
  Let $R$ be a commutative unital ring, $a \in R$ and $n \in \n$. The \emph{Temperley--Lieb algebra $\tl_{n}(a)$ with parameter $a \in R$} is the $R$-algebra with generators $U_1, \dots, U_{n-1}$ and the following relations
  \begin{enumerate}
  \item $U_iU_j = U_jU_i$, if $|i - j| \geq 2$.
  \item $U_iU_jU_i = U_i$, if $j = i \pm 1$.
  \item $U_i^2 = a U_i$ for all $i$.
  \end{enumerate}
  The unit $1$ corresponds to the empty product of the generators $U_i$. Note that $$\tl_0(a) = \tl_1(a) = R.$$
\end{definition}

Jones \cite{jones1985apolynomialinvariantforknotsviavonneumannalgebras} used Temperley--Lieb algebras to define the polynomial invariant for knots, which we now call the Jones polynomial. The following diagrammatic interpretation is due to Kauffman \cite{kauffman1987statemodelsandthejonespolynomial, kauffman1990aninvariantofregularisotopy}. We follow \cite{kasselturaev2008braidgroups}, Section 5.7.4, in our exposition.

\begin{definition}
  A \emph{planar diagram of $n \geq 1$ arcs $D = \{\gamma_1, \dots, \gamma_n\}$} in $[0,1] \times \mathbb{R}$ is a disjoint union of $n$ smoothly embedded arcs $\gamma_i: [0,1] \to [0,1] \times \mathbb{R}$ such that:
  \begin{enumerate}
  \item The images of any two arcs $\gamma_i$ and $\gamma_j$ are pairwise disjoint.
  \item The points $\gamma_i(0)$ and $\gamma_i(1)$ are a subset of the points $$\{(0,1), \dots, (0,n), (1,1), \dots, (1,n)\}.$$
  \item The tangent vectors at $\gamma_i(0)$ and $\gamma_i(1)$ are parallel to the x-axis $\mathbb{R} \times 0$.
  \end{enumerate}
\end{definition}

Let $a \in R$ and $P_n(a)$ be the free $R$-module spanned by the set of isotopy classes $[D]$ of planar diagrams $D$. We will now explain how the module $P_n(a)$ can be equipped with the structure of an associative $R$-algebra. Given two isotopy classes of planar diagrams $[D]$ and $[D']$, we obtain a diagram in $[0,1] \times \mathbb{R}$ by pasting $D$ into $[0,\frac{1}{2}] \times \mathbb{R}$ and $D'$ into $[\frac{1}{2}, 1] \times \mathbb{R}$. We can choose representatives $D \in [D]$ and $D' \in [D']$, such that the resulting diagram consists of a planar diagram $D \circ D'$ and $k(D,D') \geq 0$ circles. The product of $[D]$ and $[D']$ is defined as
$$[D] \cdot [D'] = a^{k(D,D')}[D \circ D'].$$
The reader is invited to revisit \cref{fig:temperleyliebmultiplication}, which illustrates this definition.

\begin{theorem}[{\cite[Theorem 5.34]{kasselturaev2008braidgroups}, \cite[Section 2]{boydhepworth2020temperleylieb}}] \label{diagraminterpretation}
  The Temperley--Lieb algebra $\tl_n(a)$ is isomorphic to the $R$-algebra of planar diagrams $P_n(a)$. The isomorphism is given by mapping a generator $U_i \in \tl_n(a)$ to the isotopy class of the following planar diagram.
  \begin{figure}[h]
    \centering
    \def\svgscale{.9}
\begingroup%
  \makeatletter%
  \providecommand\color[2][]{%
    \errmessage{(Inkscape) Color is used for the text in Inkscape, but the package 'color.sty' is not loaded}%
    \renewcommand\color[2][]{}%
  }%
  \providecommand\transparent[1]{%
    \errmessage{(Inkscape) Transparency is used (non-zero) for the text in Inkscape, but the package 'transparent.sty' is not loaded}%
    \renewcommand\transparent[1]{}%
  }%
  \providecommand\rotatebox[2]{#2}%
  \newcommand*\fsize{\dimexpr\f@size pt\relax}%
  \newcommand*\lineheight[1]{\fontsize{\fsize}{#1\fsize}\selectfont}%
  \ifx\svgwidth\undefined%
    \setlength{\unitlength}{240.94488189bp}%
    \ifx\svgscale\undefined%
      \relax%
    \else%
      \setlength{\unitlength}{\unitlength * \real{\svgscale}}%
    \fi%
  \else%
    \setlength{\unitlength}{\svgwidth}%
  \fi%
  \global\let\svgwidth\undefined%
  \global\let\svgscale\undefined%
  \makeatother%
  \begin{picture}(1,0.58823529)%
    \lineheight{1}%
    \setlength\tabcolsep{0pt}%
    \put(0,0){\includegraphics[width=\unitlength,page=1]{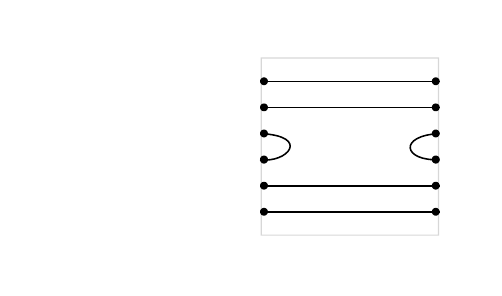}}%
    \put(0.88756916,0.15679681){\color[rgb]{0,0,0}\makebox(0,0)[lt]{\lineheight{1.25}\smash{\begin{tabular}[t]{l}\small $1$\end{tabular}}}}%
    \put(0.88775854,0.2608026){\color[rgb]{0,0,0}\makebox(0,0)[lt]{\lineheight{1.25}\smash{\begin{tabular}[t]{l}\small $i$\end{tabular}}}}%
    \put(0.88756916,0.41669707){\color[rgb]{0,0,0}\makebox(0,0)[lt]{\lineheight{1.25}\smash{\begin{tabular}[t]{l}\small $n$\end{tabular}}}}%
    \put(0.69620985,0.37954551){\color[rgb]{0,0,0}\makebox(0,0)[lt]{\lineheight{1.25}\smash{\begin{tabular}[t]{l}$\vdots$\end{tabular}}}}%
    \put(0.69620985,0.17248662){\color[rgb]{0,0,0}\makebox(0,0)[lt]{\lineheight{1.25}\smash{\begin{tabular}[t]{l}$\vdots$\end{tabular}}}}%
    \put(0.06025853,0.29328372){\color[rgb]{0,0,0}\makebox(0,0)[lt]{\lineheight{1.25}\smash{\begin{tabular}[t]{l}$U_i \hspace{40pt} \mapsto$\end{tabular}}}}%
  \end{picture}%
\endgroup%

    \caption{Translating between algebraic and diagrammatic definition of $\tl_n(a)$.}
    \label{fig:algebraicvsdiagrammaticdefinition}
  \end{figure}
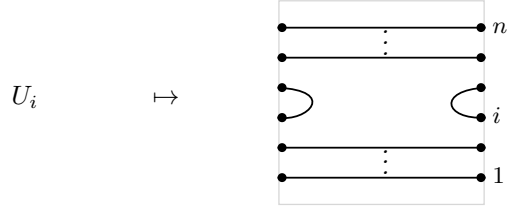
\end{theorem}

\begin{definition} \label{planardiagram}
  An element $A \in \tl_n(a)$ is called a \emph{planar diagram} in $\tl_n(a)$, if the isomorphism in \cref{diagraminterpretation} identifies $A$ with the isotopy class $[D] \in P_n(a)$ of some planar diagram $D$ of $n$ arcs.
\end{definition}

Let $\widehat{\tl}_n$ be the two-sided ideal of $\tl_n = \tl_n(a)$ generated by the set $$\{U_i : 1 \leq i \leq n-1\}.$$ In the diagrammatic interpretation (see \cref{diagraminterpretation}),  $\widehat{\tl}_n$ is spanned by all planar diagrams except the identity diagram.

\begin{definition}\label{tltrivialmodule}
  The \emph{trivial module $\trivial$} of the Temperley--Lieb algebra $\tl_n = \tl_n(a)$ is defined via the following exact sequence
  $$0 \to \widehat\tl_n \hookrightarrow \tl_n \twoheadrightarrow \trivial \to 0$$
\end{definition}

Equivalently, the trivial module $\trivial$ is the one-dimensional $\tl_n$-module coming from the augmentation $\epsilon: \tl_n \to R$ that sends every generator $U_i$ to zero.

\section{Vanishing theorems for the homology of Temperley--Lieb algebras} \label{sec:vanishing-theorems-for-the-homology-of-tl-algebras}

This section builds on ideas contained in \cite{boydhepworth2020temperleylieb} and \cite{boyd2020thelowdimensionalhomology}. We introduce the cellular Davis complex for Temperley--Lieb algebras $\tl\frakd$ and prove that it is contractible. We then use it to derive the vanishing results for the homology of Temperley--Lieb algebras, which we stated as Theorem A, B and C in the introduction.

Throughout this section, let $\tl = \tl_{n}(a)$ denote the Temperley--Lieb algebra on $n$ strands with parameter $a \in R$ and let $U_1, \dots, U_{n-1}$ be the standard generators. We sometimes identify $\{U_1, \dots, U_{n-1}\}$ with the set $\setn = \{1, \dots, n-1\}$.

\begin{definition} \label{innermost} \hspace{1pt}
  \begin{enumerate}
  \item A (possibly empty) subset $F \subseteq \setn$ is called \emph{innermost}, if for any $$i \neq j \in F: |i - j| \geq 2.$$
  \item Let $A \in \tl$ be a planar diagram in the sense of \cref{planardiagram} with corresponding isotopy class $[A] \in P_{n}(a)$. $A \in \tl$ is represented by certain monomials in the generating set $\{U_1, \dots, U_{n-1}\}$. We write $$F(A) = \{U_{i_k} : A = U_{i_1} \cdot \dots \cdot U_{i_k} \}$$ for the set of possible last letters in a monomial representing $A$. Using the identification in \cref{diagraminterpretation}, the set $F(A)$ has the following diagrammatic interpretation $$F(A) = \{U_i : [A] \text{ has an arc connecting } (1,i) \text{ and } (1,i+1)\}.$$
    We call an arc connecting $(1,i)$ and $(1,i+1)$ an \emph{innermost right cup} at position $i$ and $F(A)$ \emph{the set of innermost right cups} of the planar diagram $A \in \tl$. Note that $F(A)$ is an innermost set.
  \end{enumerate}  
\end{definition}

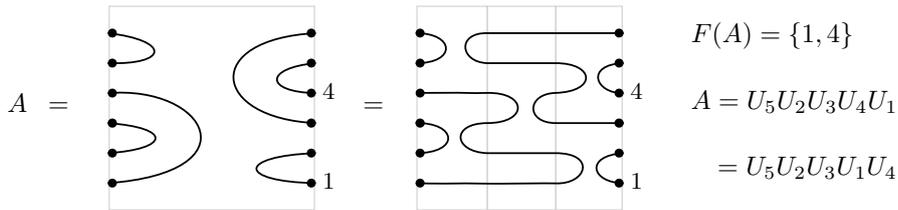
\begin{figure}[ht]
  \centering
  \def\svgscale{.9}
\begingroup%
  \makeatletter%
  \providecommand\color[2][]{%
    \errmessage{(Inkscape) Color is used for the text in Inkscape, but the package 'color.sty' is not loaded}%
    \renewcommand\color[2][]{}%
  }%
  \providecommand\transparent[1]{%
    \errmessage{(Inkscape) Transparency is used (non-zero) for the text in Inkscape, but the package 'transparent.sty' is not loaded}%
    \renewcommand\transparent[1]{}%
  }%
  \providecommand\rotatebox[2]{#2}%
  \newcommand*\fsize{\dimexpr\f@size pt\relax}%
  \newcommand*\lineheight[1]{\fontsize{\fsize}{#1\fsize}\selectfont}%
  \ifx\svgwidth\undefined%
    \setlength{\unitlength}{382.67716535bp}%
    \ifx\svgscale\undefined%
      \relax%
    \else%
      \setlength{\unitlength}{\unitlength * \real{\svgscale}}%
    \fi%
  \else%
    \setlength{\unitlength}{\svgwidth}%
  \fi%
  \global\let\svgwidth\undefined%
  \global\let\svgscale\undefined%
  \makeatother%
  \begin{picture}(1,0.37037037)%
    \lineheight{1}%
    \setlength\tabcolsep{0pt}%
    \put(0,0){\includegraphics[width=\unitlength,page=1]{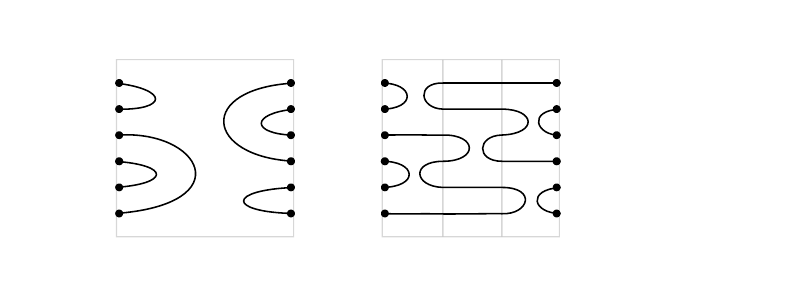}}%
    \put(0.37721169,0.0965831){\color[rgb]{0,0,0}\makebox(0,0)[lt]{\lineheight{1.25}\smash{\begin{tabular}[t]{l}\small $1$\end{tabular}}}}%
    \put(0.37721169,0.1948107){\color[rgb]{0,0,0}\makebox(0,0)[lt]{\lineheight{1.25}\smash{\begin{tabular}[t]{l}\small $4$\end{tabular}}}}%
    \put(0.7105006,0.0965831){\color[rgb]{0,0,0}\makebox(0,0)[lt]{\lineheight{1.25}\smash{\begin{tabular}[t]{l}\small $1$\end{tabular}}}}%
    \put(0.7105006,0.1948107){\color[rgb]{0,0,0}\makebox(0,0)[lt]{\lineheight{1.25}\smash{\begin{tabular}[t]{l}\small $4$\end{tabular}}}}%
    \put(0.03595246,0.18044111){\color[rgb]{0,0,0}\makebox(0,0)[lt]{\lineheight{1.25}\smash{\begin{tabular}[t]{l}$A \hspace{5pt} =$\end{tabular}}}}%
    \put(0.42113762,0.18044111){\color[rgb]{0,0,0}\makebox(0,0)[lt]{\lineheight{1.25}\smash{\begin{tabular}[t]{l}$=$\end{tabular}}}}%
    \put(0.77615092,0.25627237){\color[rgb]{0,0,0}\makebox(0,0)[lt]{\lineheight{1.25}\smash{\begin{tabular}[t]{l}$F(A) = \{1,4\}$\\\\$A = U_5U_2U_3U_4U_1$\\\\$\hspace{10pt}= U_5U_2U_3U_1U_4$\\\end{tabular}}}}%
  \end{picture}%
\endgroup%

  \caption{The set of innermost right cups of an element in  $\tl_6$.}
  \label{fig:innermostrightcupsintl6}
\end{figure}

Observe that by relation $i)$ in \cref{temperleyliebalgebras}, all generators in an innermost set commute with each other. The following modules are therefore well-defined.

\begin{definition} \label{cupmodules}
  Let $F \subseteq \setn$ be an innermost set. Then, we write $\cupmod(F)$ for the left submodule of $\tl$ generated by $\prod_{i \in F} U_i$. Using \cref{diagraminterpretation}, this is the $\tl$-submodule of $\tl_{n}(a) \cong P_{n}(a)$ spanned by all isotopy classes of planar diagrams $[A] \in P_{n}(a)$ that have an innermost right cup at position $i$ for any $i \in F$, i.e.\ that satisfy $F \subseteq F(A)$.
\end{definition}

\subsection{The cellular Davis complex of a Temperley--Lieb algebra} \label{subsec:the-cellular-davis-complex-of-tl} Using the definitions above, we will now introduce and study the $\tl$-chain complex that enables us to formulate our proofs for Theorem A, B and C.

\begin{definition} \label{cellulardaviscomplextl}
  The \emph{cellular Davis complex} $(\tl\frakd, \delta)$ of $\tl$ is the chain complex whose chain module in degree $\alpha$ is given by $$\tl\frakd_\alpha = \bigoplus_{\substack{F \subseteq \setn \text{ innermost},\\ |F| = \alpha}} \cupmod(F)$$ and whose differential
$$\delta_\alpha : \tl\frakd_\alpha \to \tl\frakd_{\alpha-1}$$
factorizes summand-wise as 
$$ \cupmod(F) \to \oplus_{s \in F} \cupmod(F_s) \hookrightarrow \tl\frakd_{\alpha-1},$$
where $F_s = F - \{s\}$ for $s \in F$. The first arrow in this factorization is the map
$$ A \mapsto \sum_{s \in F} (-1)^{\gamma_F(s)} \iota_F^{F_s} (A),$$
where $\iota_F^{F_s}: \cupmod(F) \hookrightarrow \cupmod(F_s)$ is the inclusion and $\gamma_F(s) = |\{s' \in F: s' < s\}|$.
\end{definition}

In other words, the chain module $\tl\frakd_\alpha$ of the cellular Davis complex $\tl\frakd$ has a $R$-module basis consisting of planar diagrams $A$ with $\alpha = |F|$ marked innermost right cups $F \subseteq F(A)$ (see \cref{cupmodules}) and the value of the differential $\delta: \tl\frakd_\alpha \to \tl\frakd_{\alpha - 1}$ on such a marked planar diagram $(A,F)$ is an alternating sum of the marked planar diagrams $(A, F_s)$, i.e.\ copies of the same planar diagram $A$ with one mark $s \in F$ erased.

\begin{figure}[h]
  \centering
  \def\svgscale{.9}
  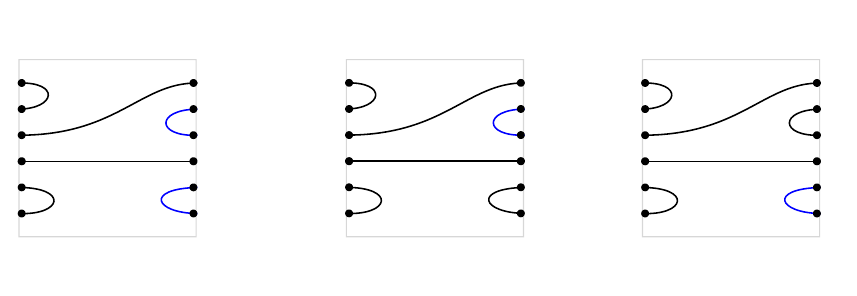
  \caption{The differential of $(\tl_6\frakd, \delta)$ evaluated on an element in $\cupmod(\{1,4\})$.}
  \label{fig:differentialoftlcomplexevaluated}
\end{figure}

We start by verifying that this really defines a chain complex. This argument is standard.

\begin{lemma}
  The cellular Davis complex $(\tl\frakd, \delta)$ is a chain complex of left $\tl$-modules.
\end{lemma}

\begin{proof}
  Let $A \in \cupmod(F) \subset \tl\frakd_\alpha$ be a planar diagram. We need to argue that $$\delta^2(A) = 0 \in \tl\frakd_{\alpha-2}.$$
  Observe that
  \begin{align*}
    \delta^2(A) &= \delta(\sum_{s \in F} (-1)^{\gamma_F(s)} \iota_F^{F_s} (A))\\
    &= \sum_{s \in F} (-1)^{\gamma_F(s)} (\sum_{t \in F_s} (-1)^{\gamma_{F_s}(t)} \iota_{F}^{F_{s,t}}(A))\\
    &= \sum_{\substack{(s,t) \in F \times F\\s \neq t}} (-1)^{\gamma_F(s)+\gamma_{F_s}(t)} \iota_F^{F_{s,t}}(A),
  \end{align*}
  where $\iota_F^{F_{s,t}}: \cupmod(F) \hookrightarrow \cupmod(F_{s,t})$ denotes the inclusion. Clearly, $F_{s,t} = F_{t,s}$. It therefore suffices to show that if $s < t$, then
  $$(-1)^{\gamma_F(s)+\gamma_{F_s}(t)}+(-1)^{\gamma_F(t)+\gamma_{F_t}(s)} = 0.$$
  This holds because $s < t$ implies that $$\gamma_F(t) = \gamma_{F_s}(t) + 1 \text{ and } \gamma_F(s) = \gamma_{F_t}(s).$$
  The $\tl$-equivariance of the differential $\delta$ follows from the $\tl$-equivariance of the inclusion maps.
\end{proof}

\begin{theorem} \label{homologycellulardaviscomplex}
  The cellular Davis complex $(\tl\frakd, \delta)$ is contractible with $$H_0(\tl\frakd, \delta) = \trivial.$$
\end{theorem}

\begin{proof}
  Let $A \in \tl$ be a planar diagram and let $F \subseteq \setn$ be an innermost set. Recall from \cref{cupmodules} that
  $$A \in \cupmod(F) \Longleftrightarrow F \subseteq F(A).$$
  Let $R \{(A,F)\} \subseteq \cupmod(F)$ denote the $R$-linear summand spanned by the isotopy class of $A$ in the diagrammatic picture using the convention that $R \{(A,F)\} = 0$ if $F \not\subseteq F(A)$. We think of $(A, F)$ as a planar diagram with $|F|$-marked innermost right cups. Then, the $R$-module $\cupmod(F)$ admits the following decomposition
  $$\cupmod(F) = \bigoplus_{\substack{A \in \tl,\\ \text{planar diagram}}} R \{(A,F)\}.$$
  Hence, the $R$-module $\tl \frakd_\alpha$ can be written as
  \begin{align*}
    \tl \frakd_\alpha &= \bigoplus_{\substack{F \subseteq \setn \text{ innermost},\\ |F| = \alpha}} \cupmod(F)\\
    &= \bigoplus_{\substack{F \subseteq \setn \text{ innermost},\\ |F| = \alpha}}  \bigoplus_{\substack{ A \in \tl, \\ \text{planar diagram}}} R \{(A,F)\}\\
    &= \bigoplus_{\substack{ A \in \tl, \\ \text{planar diagram}}} \bigoplus_{\substack{F \subseteq \setn \text{ innermost},\\ |F| = \alpha}}   R \{(A,F)\}\\
    &= \bigoplus_{\substack{ A \in \tl, \\ \text{planar diagram}}} \tl \frakd (A)_\alpha,
  \end{align*}
  where we define $$\tl\frakd(A)_\alpha = \bigoplus_{F \subseteq F(A), |F| = \alpha} R \{(A,F)\}.$$
  Observe that for any planar diagram $A \in \tl$ the sequence of submodules $\{\tl\frakd(A)_\alpha\}_{\alpha \in \mathbb{N}}$ forms a $R$-module subcomplex $\tl\frakd(A)$ of the cellular Davis complex $\tl\frakd$ because the partial differentials in $\tl\frakd$ are given by inclusion maps $\cupmod(F) \hookrightarrow \cupmod(F_s)$ with $F_s \subset F$ (see \cref{cellulardaviscomplextl}).
  It follows that, as a chain complex in $R$-modules, the cellular Davis complex is a coproduct\footnote{A similar splitting has been used by Boyd--Hepworth--Patzt in their work on the homology of Brauer algebras \cite{boydhepworthpatzt2020thehomologyofthebraueralgebras}, Definition 5.5.}, $$\tl\frakd = \bigoplus_{\substack{A \in \tl,\\ \text{planar diagram}}} \tl\frakd(A).$$
  
  We will calculate the homology of each subcomplex $\tl\frakd(A)$. If $A = id$, then $F(A) = \emptyset$ and $\tl\frakd(A) = R \{(id, \emptyset)\}[0]$ is concentrated in degree $0$. If $A \neq id$, then $\tl\frakd(A)$ is exactly the augmented chain complex $\tilde{C}(\Delta)[+1]$ of the simplex $\Delta$ on the vertex set $F(A)$, where the vertices are ordered using $F(A) \subset \setn$ and the chain complex is shifted up by one degree. Therefore, the homology of the complex $\tl\frakd(A)$ is zero in every degree, if $A \neq id$. The identification $H_0(\tl\frakd, \delta) = \trivial$ as a $\tl$-module holds because $im(\delta_1) = \widehat\tl$, where the right side is as in \cref{tltrivialmodule}, and $\tl\frakd_0 = \cupmod(\emptyset) = \tl$.
\end{proof}

\subsection{Homology with coefficients in \texorpdfstring{$\cupmod(F)$}{Cup-modules}} \label{subsec:homology-with-coefficients-in-cup-modules} Let $F \subseteq \setn$ be an innermost set. The homology groups $H_\star(\tl, \cupmod(F))$ of the Temperley--Lieb algebra $\tl$ with coefficients in $\cupmod(F)$ will occur on the $E_1$-page of the spectral sequences, which we will use to derive Theorem A, B and C. For this reason, we will now study the modules $\cupmod(F)$ and the homology groups $H_\star(\tl, \cupmod(F))$. We begin by collecting several important observations about innermost sets.

The following notion explains why the cellular Davis complex is ``sensitive'' to the question whether the underlying Temperley--Lieb algebra is defined on an even or an odd number of strands.

\begin{definition} \label{definitionofm}
  If $n$ is even, we call $M = \{1, 3, \dots, n-1\} \subset \setn$ the \emph{unique maximal innermost set}.
\end{definition}

Note that if $n$ is even, then $M = \{1, 3, \dots, n-1\} \subset \setn$ really is the unique innermost set of maximal cardinality. If the number of strands $n \geq 3$ is odd and in contrast to the even case, there exist multiple different innermost sets of maximal cardinality (see \cref{fig:innermostsetsofmaximalcardinalityoddvseven}).

\begin{figure}[h]
  \centering
  \def\svgscale{.9}
  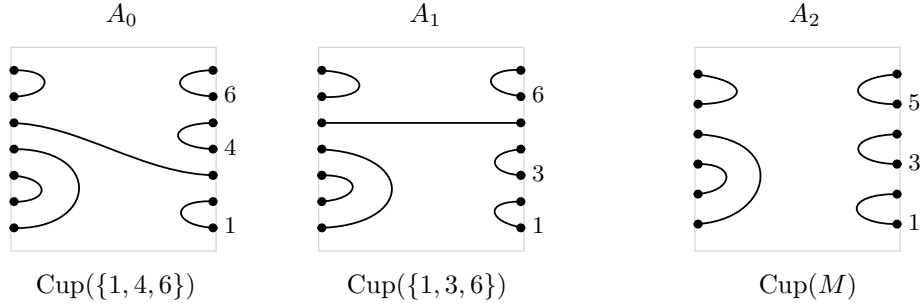
  \caption{Innermost sets of maximal cardinality for $\langle 6 \rangle$ (left) and $\langle 5 \rangle$ (right).}
  \label{fig:innermostsetsofmaximalcardinalityoddvseven}
\end{figure}

In particular the next observation always applies, if the Temperley--Lieb algebra is defined on an odd number $n \geq 3$ of strands.

\begin{observation} \label{stringnexttocup}
  If $F \subset \setn$ is nonempty and innermost, but not the unique maximal innermost set $M$, then there exists an index $1 \leq i \leq n$ such that either $i-1 \not\in F$ and $s = i+1 \in F$ or $i \not\in F$ and $s = i-2 \in F$.
\end{observation}

\begin{proof}
  Since $F$ is nonempty, there exists a smallest $z \in F$. If $z \neq 1$,  we set $1 \leq i \coloneqq z - 1 \leq n$. Then, $i - 1 = z - 2 \notin F$ by the minimality of $z$ and $s = i + 1 = z \in F$ by definition. Hence, the first case applies. Assume that $z = 1 \in F$. Because $F \neq M$, the set $\{1, \dots, n\} \setminus \{j, j+1 : j \in F\}$ is nonempty and contains a smallest index $i$. The assumption $1 \in F$ and the minimality of $i$ imply that $i \geq 3$ is an odd number such that $s = i-2 \in F$. Therefore, the second case applies.
\end{proof}

In the setting of \cref{stringnexttocup} we can think of the index $1 \leq i \leq n$ as the vertex $(1,i)$ on the right side of \emph{any} planar diagram $A$ in $\cupmod(F)$. Then, this vertex has the property that it is next to the \emph{marked} innermost right cup at position $s \in F$ but not itself the start- or endpoint of a \emph{marked} innermost right cup. \medskip

The following proposition is closely related to Section 3 of \cite{boydhepworth2020temperleylieb} and replaces the ``inductive resolutions'', which Boyd--Hepworth introduced, in our setting.

\begin{proposition} \label{retraction}
  Let $F \subset \setn$ be nonempty and innermost, but not the unique maximal innermost, and choose $i \in \{1, \dots, n\}$ and $s \in \{i-2, i+1\}$ as in \cref{stringnexttocup}. Let $F_s = F - \{s\}$. Then, $$\cupmod(F) \hookrightarrow \cupmod(F_s)$$ is a retract of left $\tl$-modules.
\end{proposition}

\begin{figure}[ht]
  \centering
  \def\svgscale{.9}
  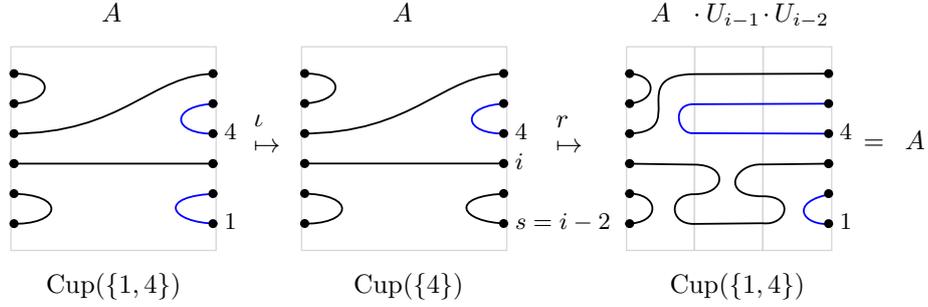
  \caption{Illustration of the retraction constructed in \cref{retraction}.}
  \label{fig:retractionofcupmodules}
\end{figure}

\begin{proof}
  Assume that $i \not\in F$ and $s = i-2 \in F$. Consider the map of left modules (see \cref{fig:retractionofcupmodules})
  $$\cupmod(F_s) \to \cupmod(F): c \mapsto c \cdot (U_{i-1}U_{i-2}).$$
  We need to check that this map is well-defined. If $c \in \cupmod(F_s)$, then $c = A \cdot \prod_{j \in F_s} U_j$ for some $A \in \tl$. Therefore, $$c \cdot (U_{i-1}U_{i-2}) = (A \cdot \prod_{j \in F_s} U_j) (U_{i-1}U_{i-2}) = (A \cdot U_{i-1}) (\prod_{j \in F} U_j),$$ where we used that $U_{i-1}$ commutes with all $\{U_z\}_{z \in F_s}$ and that $s = i-2 \in F$. The commutativity of $U_{i-1}$ with $\{U_z\}_{z \in F_s}$ follows from \cref{stringnexttocup} because
  $$|(i-1) - z| \geq \min\{|(i-1)-(i-4)|, |(i-1)-(i+1)|\} \geq 2 \text{ for any } z \in F_s.$$
  We therefore proved that $c \cdot (U_{i-1}U_{i-2}) \in \cupmod(F)$, hence that the map is well-defined. To see that it defines a retraction, we observe the following. If $c \in \cupmod(F)$, then $c = c' \cdot U_{i-2}$. Therefore, $$c \cdot (U_{i-1}U_{i-2}) = c' \cdot (U_{i-2}U_{i-1}U_{i-2}) = c' \cdot U_{i-2} = c.$$
  The argument for the case where $i-1 \not\in F$ and $s = i+1 \in F$ is similar. We consider the map of left modules
  $$\cupmod(F_s) \to \cupmod(F): c \mapsto c \cdot (U_{i} U_{i+1}).$$
  By \cref{stringnexttocup} it holds that $|i - z| \geq \min\{|i-(i-2)|, |i-(i+3)|\} \geq 2$ for $z \in F_s$. This again implies that the map is well-defined. To see that it defines a retraction, we use that if $c \in \cupmod(F)$, then $c = c' \cdot U_{i+1}$. Therefore, $$c \cdot (U_{i}U_{i+1}) = c' \cdot (U_{i+1}U_{i}U_{i+1}) = c' \cdot U_{i+1} = c.$$
\end{proof}

\begin{remark}
  The elements that we use to define the retractions in \cref{retraction} are exactly the same elements that Boyd--Hepworth use in Section 3 of \cite{boydhepworth2020temperleylieb} to define ``inductive resolutions''.
\end{remark}

The role of the following corollary in the arguments presented in this paper is similar to the role of Theorem F in \cite{boydhepworth2020temperleylieb}.

\begin{corollary} \label{vanishing}
  Let $F \subseteq \setn$ be innermost, but not the unique maximal innermost. Then, $$\cupmod(F) \hookrightarrow \tl$$ is a retract of left $\tl$-modules and $H_\star(\tl, \cupmod(F)) = 0$ for $\star > 0$.
\end{corollary}

\begin{proof}
  If $F = \emptyset$, the retraction statement is trivial because $\cupmod(\emptyset) = \tl$. If $F \neq \emptyset$, it follows from \cref{retraction} by induction. The fact that $\cupmod(F)$ is a retract of the free $\tl$-module $\tl$ implies that $\cupmod(F)$ is a projective $\tl$-module. Therefore, we conclude that $H_\star(\tl, \cupmod(F)) = 0$ for $\star > 0$.
\end{proof}

In the next lemma, we compute the degree zero homology of the Temperley--Lieb algebras with coefficients in $\cupmod(F)$.

\begin{lemma} \label{zerothhomology}
  Assume that $F \subseteq \setn$ is an innermost set. Then,
  $$H_0(\tl, \cupmod(F)) \cong \trivial \otimes_{\tl} \cupmod(F) \cong
  \begin{cases}
    R,& \text{ if } F = \emptyset,\\
    R/a,& \text{ if } n = 2 \text{ and } F = \{1\},\\
    0,& \text{ if } n > 2 \text{ and } F \neq \emptyset.
  \end{cases}
  $$
\end{lemma}

\begin{proof}
  The short exact sequence defining the trivial module $\trivial$ (\cref{tltrivialmodule}) and the right exactness of $- \otimes_{\tl} \cupmod(F)$ yield an exact sequence
  $$\widehat\tl  \otimes_{\tl} \cupmod(F) \to \tl  \otimes_{\tl} \cupmod(F) \to \trivial  \otimes_{\tl} \cupmod(F) \to 0.$$
  It follows that
  $$\trivial \otimes_{\tl} \cupmod(F) \cong \frac{\cupmod(F)}{\widehat\tl \cdot \cupmod(F)}.$$
  We use this to compute $\trivial \otimes_{\tl} \cupmod(F)$ in each case. If $F = \emptyset$, then $\cupmod(F) = \tl$ and we find that $$\trivial \otimes_\tl \cupmod(F) = \trivial \otimes_\tl \tl \cong R.$$
  If $n = 2$ and $F = \{1\}$, then $\cupmod(F) = \widehat\tl = R\{U_1\}$ and $\widehat\tl \cdot \widehat\tl = aR\{U_1\}$. Hence, $$\trivial \otimes_\tl \cupmod(F) \cong R/a.$$ For the last case, we assume that $n > 2$ and $F \neq \emptyset$. We will prove that
  $$\cupmod(F) \subseteq \widehat\tl \cdot \cupmod(F).$$
  Consider an element $A \in \cupmod(F)$. The definition of $\cupmod(F)$ implies that $A = A' \cdot (\prod_{j \in F} U_j)$ for some $A' \in \tl$. We fix a generator $U_i$, whose index $i$ has the property that $i \in F \neq \emptyset$. Because $\tl$ has $n-1 > 1$ generators, it follows that $i-1$ or $i+1$ also indexes a generator $U_{i-1}$ or $U_{i+1}$. By relation $ii)$ of \cref{temperleyliebalgebras}, it follows that
  $$A' \cdot (U_iU_{i-1}) \cdot \prod_{j \in F} U_j = A' \cdot \prod_{j \in F} U_j \text{ or } A' \cdot (U_iU_{i+1}) \cdot \prod_{j \in F} U_j = A' \cdot \prod_{j \in F} U_j.$$
  Note in either case, $B \coloneqq A' \cdot (U_iU_{i-1})$ or $B \coloneqq A' \cdot (U_iU_{i+1})$, we have that $B \in \widehat{\tl}$. Therefore,
  $$A = B \cdot (\prod_{j \in F} U_j) \in \widehat\tl \cdot \cupmod(F).$$
  It follows that $\cupmod(F) \subseteq \widehat\tl \cdot \cupmod(F)$ and hence that $\trivial \otimes_{\tl} \cupmod(F) \cong 0$.
\end{proof}

\begin{remark}
  The homology with trivial coefficients of the Temperley--Lieb algebra $\tl$ on one generator has been completely computed by Boyd--Hepworth in \cite[Proposition 7.1]{boydhepworth2020temperleylieb}. For the case $n = 2$ and $F = \{1\}$ in \cref{zerothhomology}, it holds that $H_0(\tl, \cupmod(F)) \cong H_1(\tl, \trivial)$. In particular, this case can also be deduced from \cite[Proposition 7.1]{boydhepworth2020temperleylieb}.
\end{remark}

\subsection{Proof of Theorem A, B and C} \label{subsec:proof-of-theorem-a-b-c} We will now prove the three main theorems. The following result, stated as Theorem A in the introduction, generalizes \cite[Theorem D]{boydhepworth2020temperleylieb} and the ``odd''-part of \cite[Theorem B]{boydhepworth2020temperleylieb}.

\begin{theorem} \label{theoremtlodd}
  If $n$ is odd, then $H_0(\tl, \trivial) = R$ and $H_\star(\tl, \trivial) = 0$ for $\star > 0$.
\end{theorem}

\begin{proof}
  Let $P_\star$ be a free resolution of the trivial $\tl$-module $\trivial$ and consider the double complex $P_\star \otimes_\tl \tl\frakd$, where $\tl\frakd$ is the cellular Davis complex for Temperley--Lieb algebras (see \cref{cellulardaviscomplextl}). The horizontal and vertical filtration of $P_\star \otimes_\tl \tl\frakd$ give rise to two spectral sequences. The $vE^1$-page of the vertical spectral sequence is given by $$vE^1_{\alpha,\beta} = H_\beta(P_\alpha \otimes_{\tl} \tl\frakd) \cong P_\alpha \otimes_{\tl} H_\beta(\tl\frakd).$$
  It follows from \cref{homologycellulardaviscomplex} that the $vE^1$-page is concentrated in degree $\beta = 0$, where it is given by the complex $vE^1_{\star,0} \cong P_\star \otimes_{\tl} \trivial$. By definition, we therefore have that $$vE^2_{\alpha,0} = H_\alpha(\tl, \trivial).$$
  The collapsing of the vertical spectral sequence on the $vE^2$-page implies that the horizontal spectral sequence converges to  $H_{\alpha+\beta}(\tl, \trivial)$. The $hE^1$-page of the horizontal spectral sequence is given by
  \begin{align*}
    hE^1_{\alpha, \beta} &= H_\beta(P_\star \otimes_{\tl} \tl\frakd_\alpha)\\
    &\cong \bigoplus_{\substack{F \subseteq \setn \text{ innermost},\\ |F| = \alpha}} H_\beta(P_\star \otimes_{\tl} \cupmod(F))\\
    &\cong \bigoplus_{\substack{F \subseteq \setn \text{ innermost},\\ |F| = \alpha}} H_\beta(\tl, \cupmod(F)).
  \end{align*}
  
  Because $n$ is odd, \cref{vanishing} applies to any innermost subset $F \subset \setn$. Together with \cref{zerothhomology}, this implies that
  $$hE^1_{\alpha,\beta} =
  \begin{cases}
    R,& \text{ if } (\alpha,\beta) = (0,0),\\
    0,& \text{ if } (\alpha,\beta) \neq (0,0).
  \end{cases}
  $$
  The theorem follows.
\end{proof}

The next vanishing result, stated as Theorem B in the introduction, is similar to the ``even''-part of \cite[Theorem B]{boydhepworth2020temperleylieb}. The vanishing line, which Boyd--Hepworth obtain, is stronger than ours (slope $1$ versus slope $\frac{1}{2}$) and in fact optimal \cite[Theorem C]{boydhepworth2020temperleylieb}. The description of the high-dimensional homology in terms of homology with coefficients is new in the sense that the coefficient system is different from the one occurring in \cite{boydhepworth2020temperleylieb}.

\begin{theorem} \label{theoremtleven}
  If $n$ is even, then $H_0(\tl, \trivial) = R$, $H_\star(\tl, \trivial) = 0$ for $0 < \star < \frac{n}{2}$ and $$H_{\star + \frac{n}{2}} (\tl, \trivial) \cong H_{\star}(\tl, \cupmod(M))$$ for $\star \geq 0$ and where $M = \{1,3,\dots, n-1\} \subseteq \setn$ is the unique maximal innermost set.
\end{theorem}

\begin{proof}
  By the same argument as in the proof of \cref{theoremtlodd}, we obtain a spectral sequence converging to $H_{\alpha+\beta}(\tl, \trivial)$ with $hE^1$-page $$hE^1_{\alpha, \beta} \cong \bigoplus_{\substack{F \subseteq \setn \text{ innermost},\\ |F| = \alpha}} H_\beta(\tl, \cupmod(F)).$$
  Because $n$ is even, there exists a unique innermost set $M = \{1, 3, \dots, n-1\} \subset \setn$ of maximal cardinality $|M| = \frac{n}{2}$. \cref{vanishing} applies to all innermost subsets $F \subset \setn$ except $M$. Together with \cref{zerothhomology}, this implies that
  $$hE^1_{\alpha,\beta} =
  \begin{cases}
    R,& \text{ if } (\alpha,\beta) = (0,0),\\
    H_\beta(\tl, \cupmod(M)),& \text{ if } (\alpha,\beta) = (\frac{n}{2}, \beta),\\
    0,& \text{ else.} 
  \end{cases}
  $$
  If $n = 2$, we observe that the only possibly nontrivial differential on the $hE^1$-page
  \begin{align*}
    d_1: H_0(\tl, \cupmod(M)) = \trivial \otimes_{\tl} \cupmod(M) & \to \trivial \otimes_{\tl} \tl = H_0(\tl, \tl)\\
    I \otimes A & \mapsto I \otimes A = I \cdot A \otimes 1 = 0
  \end{align*}
  is the zero map. We conclude that the spectral sequence collapses on the $hE^1$-page. The theorem follows.
\end{proof}

\begin{remark}
  \cref{zerothhomology} implies that in the setting of \cref{theoremtleven} the following holds: For $n = 2$, $H_{\frac{n}{2}}(\tl, \trivial) = R/a$, and for $n > 2$, $H_{\frac{n}{2}}(\tl, \trivial) = 0$. This is consistent with Theorem B and C of \cite{boydhepworth2020temperleylieb} mentioned above.
\end{remark}

We finish by presenting an alternative proof of \cite[Theorem A]{boydhepworth2020temperleylieb}, i.e.\ Theorem C stated in the introduction.

\begin{theorem}[{\cite[Theorem A]{boydhepworth2020temperleylieb}}] \label{homologyforinvertibleparameters}
  If the parameter $a \in R$ is a unit, then $$H_0(\tl, \trivial) = R \text{ and } H_\star(\tl, \trivial) = 0 \text{ for } \star > 0.$$
\end{theorem}

\begin{proof}
  By \cref{theoremtlodd} and \cref{theoremtleven}, it suffices to prove that for $n \geq 2$ even and $M \subset \setn$ the unique maximal innermost set, we have that $H_\star(\tl, \cupmod(M)) = 0$ for $\star \geq 0$. To see this, we consider the innermost set $M_1 = M - \{1\}$. There is a map of left $\tl$-modules $$\cupmod(M_1) \to \cupmod(M): c \mapsto c \cdot U_1$$ Observe that the map obtained by precomposition with the inclusion $$\cupmod(M) \hookrightarrow \cupmod(M_1) \to \cupmod(M)$$ is multiplication by $a$. If $a$ is a unit, the inclusion induced map $$H_\star(\tl, \cupmod(M)) \to H_\star(\tl, \cupmod(M_1))$$ must therefore be an injection. By \cref{vanishing}, the target of this map is zero if $\star > 0$. For $\star = 0$, we invoke \cref{zerothhomology}.
\end{proof}

\emergencystretch=1em
\printbibliography

\end{document}